\documentclass[draft]{birkjour}
\usepackage[noadjust]{cite}
\usepackage{xcolor}
\RequirePackage[all]{xy}

\newtheorem{thm}{Theorem}[section]

\newtheorem{prop}[thm]{Proposition}
\theoremstyle{definition}

\theoremstyle{remark}

\numberwithin{equation}{section}

\usepackage{hyperref}

\newcommand{\BibTeX}{B\kern-0.1emi\kern-0.017emb\kern-0.15em\TeX}
\newcommand{\XYpic}{$\mathrm{X\kern-0.3em\raisebox{-0.18em}{Y}}$-$\mathrm{pic}\,$}

\newcommand{\cl}{C \kern -0.1em \ell}  

\newcommand{\ed}{\end{document}}
\newcommand{\hypergeom}[5]{\mbox{$
_{#1} F_{#2}
\!\!
\left(
\!\!\!\!
\begin{array}{c}
\multicolumn{1}{c}{\begin{array}{c}
#3
\end{array}}\\[1mm]
\multicolumn{1}{c}{\begin{array}{c}
#4
            \end{array}}\end{array}
\!\!\!\!
; \displaystyle{#5}\right)
$}}
\everymath{\displaystyle}
\begin{document}

%
%
%
%
%
%
%
%
%

\title[Expansion Formulas for Brenke Polynomials]{Some Expansion Formulas for Brenke Polynomial Sets}
\author[H. Chaggara]{Hamza Chaggara}
\address{%
Mathematics
Department, College of Science, King Khalid University, Abha, Kingdom of Saudi Arabia/D\'{e}partement de Math\'{e}matiques,
\'{E}cole Sup\'{e}rieure des Sciences et de la Technologie, Sousse University, Tunisia.}
\email{hshaggara@kku.edu.sa / hamza.chaggara@ipeim.rnu.tn}
\author[A. Gahami]{Abdelhamid Gahami}
\address{%
D\'{e}partement de Math\'{e}matiques,
Institut Pr\'{e}paratoire aux \'{E}tudes d'Ing\'{e}nieur, Sfax University, 
Tunisia.}
\email{aelgahami@yahoo.fr}
\author[N. Ben Romdhane]{Neila Ben Romdhane}
%
\address{%
D\'{e}partement de Math\'{e}matiques,
\'{E}cole Sup\'{e}rieure des Sciences et de la Technologie, Sousse University, Tunisia.}
\email{neila.benromdhane@ipeim.rnu.tn}
%
\date{\today}
\dedicatory{Last Revised:\\ \today}
\begin{abstract}
In this paper, we derive some explicit expansion formulas associated to Brenke polynomials using operational rules based on their corresponding generating functions. The obtained coefficients are expressed either in terms of finite double sums or finite sums or sometimes in closed hypergeometric terms. The derived results are applied to Generalized Gould-Hopper polynomials and Generalized Hermite polynomials introduced by Szeg\"{o} and Chihara. Some well-known duplication and convolution formulas are deduced as particular cases.
\end{abstract}
\label{page:firstblob}
\maketitle
\tableofcontents
\section{Introduction}
Let $\mathcal{P}$ be the vector space of polynomials with coefficients in $\mathbb{C}$. A polynomial sequence in $\mathcal{P}$ is called \textit{polynomial set} (PS for short) if $\deg P_n=n$, for all $n$.
\\
The connection and linearization problems are defined as follows. \\
Given two PSs $\{P_n\}_{n\geq0}$ and 
$\{Q_n\}_{n\geq0}$, 
the so-called \emph{connection problem} between them asks to find the coefficients
   $C_m(n)$, called connection coefficients CC, in the expression
\begin{equation}\label{coeff connection}
   Q_n(x)=\sum_{m=0}^nC_m(n)P_m(x).
\end{equation}
The particular cases $Q_n(x)=x^n$ and  $Q_n(x)=P_n(ax), \ a\neq 0$, in (\ref{coeff connection}) are known, respectively, as the \textit{inversion formula} for $\{P_n\}_{n\geq0}$ and the \textit{duplication or multiplication formula} associated with $\{P_n\}_{n\geq0}$.

Given three PSs 
$\{P_n\}_{n\geq0}$, $\{R_n\}_{n\geq0}$ 
and $\{S_n\}_{n\geq0}$, then for \\
$Q_{i+j}(x)=~R_i(x)S_j(x)$ in (\ref{coeff connection}) we are faced to the general \emph{linearization problem}
\begin{equation}\label{linearization coeff}
   R_i(x)S_j(x)=\sum_{k=0}^{i+j}L_{ij}(k)P_k(x).
\end{equation}
The coefficients $L_{ij}(k)$ are called 
linearization coefficients LC.
\\
The particular case of this problem,    $P_n=R_n=S_n$, is known as the \textit{standard linearization problem} or
 \textit{Clebsch-Gordan-type problem}.
 
The computation and the positivity of the aforementioned coefficients play important roles in many situations of pure and applied mathematics ranging from combinatorics and statistical mechanics to group theory \cite{askey1971,gasper1970,koornwinder1994}. 
Therefore, different methods have been developed in the literature and 
several sufficient conditions for the sign properties to hold have been derived in \cite{Askey,szwarc1992}, using for this purpose specific properties of the involved polynomials such as orthogonality, generating functions, inversion formulas, hypergeometric expansion formulas, recurrence relations, algorithmic approaches, inverse relations,\ldots (see e.g.\cite{abdelhamid21,navima2015,benromdhane2016,chaggara2022b,
maroni2013,mama2015}). In particular, a general method, based on operational rules and generating functions, was developed for polynomial sets with equivalent lowering operators and with Boas-Buck 
generating 
functions~\cite{chaggara2005a,chaggara2007a,
chaggara2011}.

In this paper, we deeply discuss both the connection and the linearization problems when the involved polynomials are of Brenke type. These polynomials are defined by their exponential generating functions as follows \cite{brenke45,chi}
\begin{equation}\label{form brenke}
A(t)B(xt)=\sum_{n=0}^{\infty}\frac{P_n(x)}{n!}t^n,
\end{equation}
where $A$ and $B$ are two formal power series satisfying:
\begin{equation}\label{expression A et B}
A(t)=\sum_{k=0}^{\infty}a_kt^k,\quad B(t)=\sum_{k=0}^{\infty}b_kt^k,\quad a_0b_k\neq0, \ \forall k\in\mathbb{N}.
\end{equation}
Brenke PSs are reduced to Appell ones when $B=\exp$ and they generated many well-known polynomials in the literature, namely 
monomials, Hermite, Laguerre, Gould-Hopper, Generalized Hermite,
Generalized Gould-Hopper, Appell-Dunkl, $d$-Hermite, $d$-Laguerre, Bernoulli, Euler, Al-Salam-Carlitz, Little $q$-Laguerre, $q$-Laguerre, discrete $q$-Hermite PSs,\ldots. 

These polynomials appear in many areas of mathematics. 
In particular, in the framework of the standard orthogonality of polynomials, an exhaustive classification of all Brenke orthogonal polynomials was established by Chihara in~\cite{chihara68}. 
Furthermore, Brenke polynomials play a central role in \cite{Nobuhuro13}, 
where the authors determined all MRM-triples associated with Brenke-type generating functions.
Further, the positive approximation process discovered by Korovkin, a powerful criterion in order to decide whether a given sequence of positive linear operators on the space of continuous functions converges uniformly in this space, plays a central role and arises naturally in many problems connected with functional analysis, harmonic analysis, measure theory, partial differential equations, and probability theory.
 The most useful examples of such operators are Sz\'{a}sz operators and many authors obtained a generalization of these operators using Brenke polynomials (see \cite{varma12,wani21} and the references therein). 

This paper is organized as follows. 
In Section 2, 
we define the transfer linear operator between two Brenke polynomials and which is illustrated by three interesting examples in particular the hypergeometric transformation and the Dunkl operator on the real line. 
Then in Section 3,
 we derive expansion formulas associated to Brenke polynomials using operational rules and we give connection, linearization, inversion, duplication, and addition formulas corresponding to these polynomials. 
 The obtained coefficients are expressed using generating functions involving the associated transfer linear operators.
  Finally, in Section 4,
   we apply our obtained results to both Generalized Gould-Hopper PS (GGHPS) and Generalized Hermite PS (or Szeg\"{o}-Chihara PS) and we recover many known formulas as special cases.
\section{Operators Associated to Brenke PSs}
In this section, first, we introduce a transfer 
operator between two Brenke families, then we state its expression as an infinite series in the derivative operator $D$ and the multiplication operator $X$ known as $XD$-expansion \cite{bucchianico95}. 
Finally, we give some examples. 
\subsection{Transfer Operator Associated to two Brenke Polynomials}
Any Brenke PS $\{P_n\}_{n\geq0}$ generated by (\ref{form brenke}) is $D_b$-Appell of transfer power series $A$, where $A$ and $b=(b_n)$ are defined in (\ref{expression A et B}). That is,
\begin{equation}\label{baisc}
 D_b P_{n+1}=(n+1)P_{n}\quad \textrm{and}\quad A(D_b)(b_nx^n)=\frac{P_n}{n!},\ n=0,1,2,\ldots,
\end{equation}
where $ D_b$ denotes the linear operator on $\mathcal{P}$ defined by~\cite{chaggara2005a}:
\begin{equation}\label{def D_b}
 D_b(1)=0,\ 
 D_b(x^n)=\frac{b_{n-1}}{b_n}x^{n-1},\, n=1,2,\ldots.
\end{equation}
The operator $D_b$ is known as the lowering operator for the PS $\{P_n\}_{n\geq 0}$, however, $A$ is the associated transfer series. (For more details, see \cite{ybchmonom}).

Let $\{P_n\}_{n\geq0}$ and $\{Q_n\}_{n\geq0}$ be two Brenke PSs generated respectively by:
\begin{equation}\label{generatrice de P_n et Q_n}
A_1(t)B_1(xt)=\sum_{n=0}^\infty{P_n(x)\over n!}t^n
\quad
\textrm{ and }\quad 
A_2(t)B_2(xt)=\sum_{n=0}^\infty{Q_n(x)\over n!}t^n
,
\end{equation}
where for $i=1,2$,
\begin{equation}\label{dev of A and B}
A_i(t)=\sum_{k=0}^\infty a_k^{(i)}t^k,\quad
B_i(t)=\sum_{k=0}^\infty b_k^{(i)}t^k,
\quad  
a_0^{(i)}b_k^{(i)}\neq 0,\ \forall \, k \in \mathbb{N}.
\end{equation} 
Then, the corresponding operators $D_{b^{(1)}}$ and $D_{b^{(2)}}$ are related by:
\begin{equation}\label{relation D_b teta}
D_{b^{(2)}}\theta=\theta D_{b^{(1)}},
\end{equation}
where $\theta$ is the bijective linear operator from $\mathcal{P}$ onto $\mathcal{P}$ (isomorphism of $\mathcal{P}$) acting on monomials as follows:
\begin{equation}\label{def teta}
  \theta(x^n)=\frac{b_n^{(2)}}{b_n^{(1)}}x^n\quad
  \textrm{and}\quad\theta^{-1}(x^n)=\frac{b_n^{(1)}}{b_n^{(2)}}x^n.
\end{equation}
The linear operator $\theta$ can be extended as a transfer operator taking any formal power series to another formal power series as follows
\begin{equation}\label{theta power}
\theta(\sum_{n\geq 0}a_nx^n)=
 \sum_{n\geq 0}a_n \theta(x^n),
\end{equation}
and if $\phi(x)$ denotes a formal power series then one can easily check that,
\begin{equation}\label{thetaproperty}
\theta\Bigl(\phi(x)\sum_{k=0}^\infty a_kx^k\Bigr)=\sum_{k=0}^\infty a_k\theta(\phi(x)x^k).
\end{equation}
Hence, it is obvious that,
\begin{equation}\label{thetabb}
\theta(B_1(x))=B_2(x).
\end{equation}
The operator $\theta$ will be called the transfer operator from $B_1$ to $B_2$ or transfer operator from $\{P_n\}_{n\geq 0}$ to $\{Q_n\}_{n\geq 0}$.
\subsection{$XD$-Expansion of the  Operator $\theta$}
Now, recall that any operator $L$ acting on formal power series has the following formal expansion, known as $XD$-expansion~(see \cite{bucchianico95} and the references therein):
\begin{equation}\label{gener dev xl}
L=\sum_{k=0}^\infty A_k(X)D^k,
\end{equation}
where  $D$ denotes the ordinary differentiation operator and
 $\{A_k(x)\}_{k\geq0}$ is a polynomial sequence such that:
\begin{equation}\label{form gener xl}
Le^{xt}=\sum_{k=0}^\infty A_k(x)t^ke^{xt}.
\end{equation}
We note that the infinite sum (\ref{gener dev xl}) is always well defined on $\mathcal P$ since when applied to any given
polynomial, only a finite number of terms makes a nonzero contribution.
\\
The $XD$-expansion of the transfer operator $\theta$ is explicitly given by
\begin{prop}\label{XDexpansion}
 The operator $\theta$ defined by (\ref{def teta}) has the formal expansion:
\begin{equation}\label{xd dev de teta}
\theta=\sum_{k=0}^\infty {\phi_k\over k!}X^kD^k,
\end{equation}
where 
$\quad\phi_k=(-1)^k\sum_{m=0}^k\frac{(-k)_m}{m!}{b_m^{(2)}\over b_m^{(1)}}.
$
\end{prop}
\begin{proof} 
By using (\ref{def teta}) and (\ref{theta power}) and
 then substituting $L$ by $\theta$ in 
(\ref{form gener xl}), we obtain
\begin{equation*}
\theta(e^{xt})=
\sum_{k=0}^\infty{b_k^{(2)}\over b_k^{(1)}}{(xt)^k\over k!}=\sum_{k=0}^\infty A_k(x)t^k  e^{xt}.
\end{equation*}
Therefore,
\begin{equation*}
\sum_{k=0}^\infty A_k(x)t^k=e^{-xt}\sum_{k=0}^\infty{b_k^{(2)}\over b_k^{(1)}}{(xt)^k\over k!}=\sum_{k=0}^\infty
\left(\sum_{m=0}^k(-1)^{k}\frac{(-k)_m}{m!}{b_m^{(2)}\over b_m^{(1)}}\right){(xt)^k\over k!},
\end{equation*}
which establishes the desired result.
\end{proof}
\subsection{Examples}
Here, we consider three interesting particular cases of the linear operator $\theta$ associated to two Brenke PSs and we essentially give integral representations for this operator.
\subsubsection{Hypergeometric Transformation}
 Recall first that $_rF_s$ denotes the
  generalized hypergeometric
 function with $r$ numerator parameters and $s$ denominator parameters and defined as follows.
\begin{equation}\label{hypergeometric}
_{r}F_{s}\left(\begin{array}{l}
(\alpha_r)
\\
(\beta_s)\end{array};x\right)=
\sum_{k=0}^\infty\frac{(\alpha_1)_k(\alpha_2)_k\cdots(\alpha_r)_k}{(\beta_1)_k(\beta_2)_k
\cdots(\beta_s)_k}\frac{x^k}{k!},
\end{equation}
where the contracted notation $(\alpha_r)$ is used to abbreviate the array
\\ 
$\{\alpha_1,\ldots, \alpha_r\},$ and $(\alpha)_n$ denotes the Pochhammer symbol: 
\begin{equation}\label{ppp}
\quad(\alpha)_n=\frac{\Gamma(\alpha+n)}{\Gamma(\alpha)}.
\end{equation}
Consider two Brenke PSs $\{P_n\}_{n\geq0}$ and $\{Q_n\}_{n\geq0}$ generated by (\ref{generatrice de P_n et Q_n}) and (\ref{dev of A and B}) and such that the corresponding transfer linear operator $\theta$ takes the form:
\begin{equation}\label{frac hyp}
\theta (x^n)= 
{b_n^{(2)}\over b_n^{(1)}}x^n=
{(\gamma_1)_n(\gamma_2)_n\cdots(\gamma_p)_n\over (\delta_1)_n(\delta_2)_n\cdots(\delta_p)_n}x^n,\, 
\gamma_i\in \mathbf{C},\,\delta_i \in \mathbf{C}\setminus \{-\mathbb{N}\}  .
\end{equation}
In this case, for the action of the operator $\theta$ on hypergeometric functions, we have the following result.
\begin{prop}
Let $\theta$ be defined by (\ref{frac hyp}) with 
$0<\Re (\gamma_i)<\Re (\delta_i)$, 
then for 
$r\leq s+1$ 
and $|x|<1$, we have
\begin{align}\label{rep intg teta 2}
\theta\hypergeom{r}{s}{(\alpha_r)}{(\beta_s)}{x}&=\prod_{i=1}^p{1\over\beta(\gamma_i,\delta_i)}\int_{]0,1[^p}\prod_{i=1}^pu_i^{\gamma_i-1}(1-u_i)^{\delta_i-\gamma_i-1}\nonumber\\
                                                             &\times\hypergeom{r}{s}{(\alpha_r)}{(\beta_s)}{x\prod_{i=1}^p u_i}du_1\cdots du_p,
\end{align}
where $\beta$ designates the usual Euler's Beta function,
\begin{equation}\label{beta}
\beta(\gamma,\delta)=\int_0^1t^{\gamma-1}(1-t)^{\delta-1}dt=\frac{\Gamma(\gamma)\Gamma(\delta)}{\Gamma(\gamma+\delta)},\ \Re(\gamma),\Re(\delta)>0.
\end{equation}
\end{prop}
\begin{proof}
From (\ref{theta power}) and (\ref{frac hyp}), we have
\begin{equation*}
\theta\hypergeom{r}{s}{(\alpha_r)}{(\beta_s)}{x}=\hypergeom{{p+r}}{{p+s}}{(\alpha_r),(\gamma_p)}{(\beta_s),(\delta_p)}{x}.
\end{equation*}
Thus, by using the Euler integral representation of generalized hypergeometric functions, we obtain (see \cite[p.~85]{rain}):
\begin{align*}
\hypergeom{{p+r}}{{p+s}}{(\alpha_r),(\gamma_p)}{(\beta_s),(\delta_p)}{x}&={\Gamma(\delta_p)\over\Gamma(\gamma_p)\Gamma(\delta_p-\gamma_p)}\int_0^1u_p^{\delta_p-1}(1-u_p)^{\gamma_p-\delta_p-1}\\
&\times \hypergeom{{p+r-1}}{{p+s-1}}{(\alpha_r),(\gamma_{p-1})}
{(\beta_s),(\delta_{p-1})}
{x u_p}\,du_p,
\end{align*}
and after $(p-1)$ similar applications of the Euler integral representation we get the desired result.
\end{proof}
When the operator $\theta$ is given by (\ref{frac hyp}), the coefficient $\phi_k$ in Proposition~ \ref{XDexpansion} is  
\begin{align*}
\phi_k&=(-1)^k\sum_{m=0}^{k}(-k)_m{(\gamma_1)_m(\gamma_2)_m\cdots(\gamma_p)_m\over m!(\delta_1)_m(\delta_2)_m\cdots(\delta_p)_m}\\
&=(-1)^ki\hypergeom{p+1}{p}{-k,\gamma_1,\gamma_2,\ldots,\gamma_p}{\delta_1,\delta_2,\ldots,\delta_p}{1}.
\end{align*}
Thus the corresponding $XD$ expansion is 
\begin{equation}\label{XDexpansionhypergeom}
\theta=\sum_{k=0}^\infty {(-1)^k\over k!}\hypergeom{p+1}{p}{-k,\gamma_1,\gamma_2,\ldots,\gamma_p}{\delta_1,\delta_2,\ldots,\delta_p}{1}X^kD^k.
\end{equation}
\subsubsection{Particular Hypergeometric Transformation}
Here, we consider the special case 
$\theta(x^n)=\frac{(\gamma)_n}{(\delta)_n}x^n,\ \delta\neq 0,-1,-2,\ldots.$ 
 \begin{prop}\label{prophyper}
For any analytic function $f$ on 
$]-1,1[,\,f(x)=\sum_{n=0}^{\infty}a_nx^n$, 
we have 
\begin{equation}\label{nnn}
\theta(f)(x)=
\frac{1}{\beta(\gamma,\delta-\gamma)}
\int_0^1 t^{\gamma-1} 
(1-t)^{\delta-\gamma-1}
f(xt)dt,\  0<\Re (\gamma)<\Re(\delta).
\end{equation} 
Moreover, the $XD$-expansion of $\theta$ is the following
\begin{equation}\label{xd dev hyp geo}
  \theta=\sum_{k=0}^{\infty}{(-1)^k\over k!}\frac{(\delta-\gamma)_k}{(\gamma)_k}X^kD^k.
\end{equation}
\end{prop}
\begin{proof}
By using (\ref{ppp}) and (\ref{beta}), we obtain
$$
\frac{(\gamma)_n}{(\delta)_n}x^n
=\frac{\Gamma(\gamma+n)}{\Gamma(\delta+n)}\frac{\Gamma(\delta)}{\Gamma(\gamma)}x^n
=
\frac{1}{\beta(\gamma,\delta-\gamma)}\int_0^1 t^{\gamma-1}
(1-t)^{\delta-\gamma-1}
(xt)^ndt.
$$
Thus, substituting the above equation in (\ref{theta power}), we obtain (\ref{nnn})  
since the term-by-term integration is justified by the convergence of the series
\begin{equation*}
 \sum_{n\geq0}\int_{0}^1\left|a_n
 t^{\gamma-1}(1-t)^{\delta-\gamma-1} (xt)^n
 \right|\,dt.
\end{equation*}
For (\ref{xd dev hyp geo}), we use (\ref{XDexpansionhypergeom}) and the Chu-Vandermonde reduction formula:
\begin{equation}\label{Chu}
_{2}F_{1}\left(\begin{array}{l}
-k,\gamma\\
\delta\end{array};1\right)=\frac{(\delta-\gamma)_k}{(\delta)_k},\quad \delta\neq 0,-1,-2,\ldots.
\end{equation}
Thus the proof is completed.
\end{proof}
\subsubsection{Dunkl Operator on the Real Line}
The well-known Dunkl operator, 
$\mathcal{D}_\mu$, associated with the parameter $\mu$ on the real line provides a useful tool in the study of special functions with root systems associated with finite
reflection groups~\cite{dunkl91} and it is closely related to certain representations of degenerate affine Heke 
algebras~\cite{opdam1993}. This operator is defined by~\cite{dunkl91}:
\begin{equation}\label{def dunkl}
    \mathcal{D}_\mu(f)(x)=Df(x)+\frac{\mu}{x}(f(x)-f(-x)),\quad \mu\in\mathbb{C},
\end{equation}
where $f$ is a real variable complex-valued function and $D$ is the differentiation operator.
\\ 
The Dunkl operator acts on monomials as follows:
\begin{equation}\label{d_mu on monomial}
    \mathcal{D}_\mu(x^n)=\frac{\gamma_\mu(n)}{\gamma_\mu(n-1)}x^{n-1},\ \mu\neq-\frac{1}{2},-\frac{3}{2},\ldots,
\end{equation}
where
\begin{equation}\label{def gamma mu}
 \gamma_\mu(2p+\epsilon)=2^{2p+\epsilon}p!(\mu+\frac{1}{2})_{p+\epsilon},\quad\epsilon=0,1.
\end{equation}
Hence, $\mathcal{D}_\mu$ is a $D_b$-operator type with $b_n=\frac{1}{\gamma_\mu(n)}$, and we have the following result. 
 \begin{prop}\label{rrr1}
Let $\mu_1$ and $\mu_2$ be two real numbers satisfying
$-\frac{1}{2}<\mu_1<\mu_2$, and $\theta$ given by
\begin{equation}\label{thetadunkl}
  \theta(x^n)=\frac{\gamma_{\mu_1}(n)}{\gamma_{\mu_2}(n)}x^n.
\end{equation}
Then, for any analytic function, $f$ on $]-1,1[,$ the following integral representation of $\theta$ holds true
\begin{equation}\label{reps intg teta}
\theta(f)(x)={1\over\beta(\mu_1+\frac{1}{2},\mu_2-\mu_1)}\int_{-1}^1f(xt)|t|^{2\mu_1}(1-t)^{\mu_2-\mu_1-1}(1+t)^{\mu_2-\mu_1}\,dt.
\end{equation}
\end{prop}
\begin{proof}
By using (\ref{ppp}), (\ref{beta}) and~(\ref{def gamma mu}) with $\mu$ replaced by $\mu_1$ and $\mu_2$, 
and for 
$n=2p+\epsilon,\ \epsilon =0,1,$ 
we obtain:
\begin{equation}\label{fracdunkl}
 {\gamma_{\mu_1}(n)\over\gamma_{\mu_2}(n)}=\frac{\beta(\mu_1+\frac{1}{2}+p+\epsilon,\mu_2-\mu_1)}{\beta(\mu_1+\frac{1}{2},\mu_2-\mu_1)}.
\end{equation}
Now, with the beta integral representation (\ref{beta}), we get
$$\beta(\mu_1+\frac{1}{2}+p+\epsilon,\mu_2-\mu_1)=\int_0^1 t^{\mu_1+p+\epsilon-\frac{1}{2}}(1-t)^{\mu_2-\mu_1-1}\; dt,$$
which, after the substitution $u^2=t$, 
and the distinction of the two cases $\epsilon=0$ and $\epsilon =1$, becomes
$$\beta(\mu_1+\frac{1}{2}+p+\epsilon,\mu_2-\mu_1)=\int_{-1}^1u^n\vert u\vert^{2\mu_1}(1-\mu)^{\mu_2-\mu_1-1}(1+u)^{\mu_2-\mu_1}\; du.$$
Consequently, this gives
\begin{equation}\label{rep teta}
\theta(x^n)=
{1\over\beta(\mu_1+\frac{1}{2},\mu_2-\mu_1)}\int_{-1}^1(xt)^n|t|^{2\mu_1}(1-t)^{\mu_2-\mu_1-1}(1+t)^{\mu_2-\mu_1}\,dt,
\end{equation}
and a term-by-term integration achieves the proof.
\end{proof}
The following two particular cases are worthy to note. 
\begin{enumerate}
\item[$\bullet$] 
For $f=\exp_{\mu_1}$, and according to (\ref{thetabb}), it is clear that 
$$
\theta (\exp_{\mu_1})=\exp_{\mu_2},
$$
 where the generalized exponential function, 
 $\exp_{\mu}$ 
 is defined by~\cite{rosen}
\begin{equation}\label{exp mu}
    \exp_{\mu}(x)=\sum_{n=0}^\infty\frac{x^n}{\gamma_{\mu}(n)},\quad\mu\neq-\frac{1}{2},-\frac{3}{2},-\frac{5}{2},\ldots.
\end{equation}
So, for $-\frac{1}{2}<\mu_1<\mu_2$, 
and by virtue of (\ref{reps intg teta}), the following integral representation of $\exp_{\mu_2}$ holds true~\cite[Eq.~(2.3.4)]{rosen}:
\begin{align*}\label{integ exp}
\exp_{\mu_2}(x)&={1\over\beta(\mu_1+\frac{1}{2},\mu_2-\mu_1)}\times\\
&\int_{-1}^1\exp_{\mu_1}(xt)|t|^{2\mu_1}(1-t)^{\mu_2-\mu_1-1}(1+t)^{\mu_2-\mu_1}\, dt.
\end{align*}
\item[$\bullet$]
 For $\mu_1=0$ and $\mu_2=\mu >0$, the transfer operator $\theta$ reduces to the well-known Dunkl intertwining operator $V_{\mu}$ in the one dimensional case and 
(\ref{reps intg teta}) is nothing else that its corresponding integral representation~\cite[Theorem 5.1]{dunkl91}:
\begin{equation}\label{reps intg vmu}
V_{\mu}(f)(x)={1\over\beta(\frac{1}{2},\mu)}\int_{-1}^1f(xt)(1-t)^{\mu-1}(1+t)^{\mu}\, dt.
\end{equation}
\end{enumerate}
\section{Connection and Linearization Problems}
In this section, we investigate connection and linearization formulas for \\
Brenke PSs.
\subsection{Connection Problem}
 Next, for two polynomial sequences of Brenke type, we state a generating function for the connection coefficients using the operator $\theta$.
This result appears to be new. Some applications are given.
   \begin{thm}\label{thm gener c.c}
     Let $\{P_n\}_{n\geq0}$ and $\{Q_n\}_{n\geq0}$ be two polynomial sequences generated by~(\ref{generatrice de P_n et Q_n}) and~(\ref{dev of A and B}) 
     and let $\theta$ be the corresponding transfer operator defined in (\ref{def teta}). 
Then the CC in (\ref{coeff connection}), $(C_m(n))_{n\geq m\geq 0}$, are generated by:
\begin{equation}\label{generatrice cc}
A_2(t) \theta\left({t^m \over A_1(t)}\right)=
 \sum_{n=m}^\infty\frac{m!}{n!}C_m(n)t^n
 .
\end{equation}
\end{thm}
\begin{proof}
On one hand, substituting (\ref{coeff connection}) in (\ref{generatrice de P_n et Q_n}) and using sum manipulations, we get:
\begin{align*}
    A_2(t)B_2(xt) &= \sum_{n=0}^\infty Q_n(x){t^n\over n!}
     = \sum_{n=0}^\infty\left(\sum_{m=0}^nC_m(n)P_m(x)\right){t^n\over n!} \\
     &= \sum_{m=0}^\infty\left(\sum_{n=m}^\infty{m!\over n!}C_m(n)t^n\right){P_m(x)\over m!}.
\end{align*}
On the other hand, from (\ref{thetaproperty}), we have 
\begin{align*}
  A_2(t)B_2(xt) &= A_2(t)\theta_tB_1(xt)= A_2(t)\theta_t\left({1\over A_1(t)}\sum_{m=0}^\infty P_m(x){t^m\over m!}\right) \\
   &= \sum_{m=0}^\infty A_2(t)\theta_t\left({t^m\over A_1(t)}\right){P_m(x)\over m!}.
\end{align*}
Thus (\ref{generatrice cc}) follows and the proof is completed.
\end{proof}
Some known results can be deduced from Theorem \ref{generatrice cc}. Next, we quote the four important ones of them.
\subsubsection{Explicit Expression of the Connection Coefficients}
\, \\
Write 
$\frac{1}{A_1(t)}=\sum_{n=0}^\infty \widehat{a}_n^{(1)}t^n,$ 
 then 
$$
\theta_t\left(\frac{t^m}{A_1(t)}\right)=\sum_{n=0}^\infty{b_{n+m}^{(2)}\over b_{n+m}^{(1)}}\widehat{a}_n^{(1)}t^{n+m}.
$$
By virtue of (\ref{generatrice cc}), we get:
  \begin{align*}
    \sum_{n=m}^\infty {m!\over n!}C_m(n)t^n &= \left(\sum_{n=0}^\infty a_n^{(2)}t^n\right)\left(\sum_{n=0}^\infty{b_{n+m}^{(2)}\over b_{n+m}^{(1)}}
\widehat{a}_n^{(1)}t^{n+m}\right) \\
     &= t^m\sum_{n=0}^\infty\left(\sum_{k=0}^n a_k^{(2)}{b_{n+m-k}^{(2)}\over b_{n+m-k}^{(1)}}\widehat{a}_{n-k}^{(1)}\right)t^n \\
     &= \sum_{n=m}^\infty\left(\sum_{k=0}^{n-m}{b_{n-k}^{(2)}\over b_{n-k}^{(1)}}a_k^{(2)}\widehat{a}_{n-m-k}^{(1)}\right)t^n.
  \end{align*}
Thus,
\begin{equation}\label{ex explicit c.c}
C_m(n)={n!\over m!}\sum_{k=0}^{n-m}{b_{n-k}^{(2)}\over b_{n-k}^{(1)}}a_k^{(2)}\widehat{a}_{n-m-k}^{(1)},\quad m=0,\ldots,n.
\end{equation}
In particular, we can deduce the explicit expansion and the inversion formula for any Brenke PS $\{P_n\}_{n\geq 0}$ generated by (\ref{form brenke}):
\begin{equation}\label{explicit-inversion Brenke}
\frac{P_n(x)}{n!}=\sum_{m=0}^n  b_ma_{n-m}x^m,\quad \textrm{and}\quad 
b_n x^n=\sum_{m=0}^n  \widehat{a}_{n-m}\frac{P_m(x)}{m!}.
\end{equation}
\subsubsection{Connection between two $D_b$-Appell PSs}
If $B_1=B_2$, 
in (\ref{generatrice de P_n et Q_n}), then by using (\ref{def teta}), we obtain that 
the expression (\ref{generatrice cc}) takes the following simpler form~\cite{chaggara2007c}.
\begin{equation}\label{c.c two equivalent}
{A_2(t)\over A_1(t)}=   \sum_{n=m}^\infty\frac{m!}{n!}C_m(n)t^{n-m}.
\end{equation}
\subsubsection{Addition and Convolution Type Formulas}
The Brenke PS $\{P_n\}_{n\geq0}$ generated by (\ref{form brenke}) possesses the following generalized addition formula and convolution type relation:
\begin{equation*}\label{generalizd trans}
  T_y^bP_n(x)=\sum_{m=0}^n\frac{n!}{m!}b_{n-m}y^{n-m}P_m(x),
\end{equation*}
and 
\begin{equation*}\label{generalizd convol}
 A(D_b)T_y^bP_n(x)=\sum_{m=0}^n{n\choose m}P_{n-m}(y)P_m(x),
\end{equation*}
where $T_y^b=B(yD_b)$ designates the generalized translation operator 
satisfying $T_y^b(B(xt)=B(yt)B(xt)$.

In fact, for the addition formula, we remark that the PS, $\{T_y^bP_n(x)\}_{n\geq 0}$, is generated by:
\begin{equation*}
    B(yt)A(t)B(xt)=\sum_{n=0}^{\infty}{T_y^bP_n(x)\over n!}t^n,
\end{equation*}
then we apply (\ref{c.c two equivalent}) with $A_2(t)=B(yt)A(t)$ and $A_1(t)=A(t)$,
to obtain
 $$
 C_m(n)=\frac{n!}{m!}b_{n-m}y^{n-m}.
 $$
 
For the convolution type relation, we apply the operator $A(D_b)$ to each member of the addition formula and we use (\ref{baisc}). We have
 \begin{align*}
      A(D_b)T_y^bP_n(x) &= \sum_{m=0}^{n}{n!\over m!(n-m)!}A(D_b)((n-m)!b_{n-m}y^{n-m})P_m(x)  
       \\
    &= \sum_{m=0}^{n}{n\choose m}P_{n-m}(y)P_m(x).
 \end{align*}
\subsubsection{Duplication Formula}
 Brenke PS generated by (\ref{form brenke}) possesses the following duplication formula \cite{chaggara2007c}
\begin{equation}\label{dupli formulae}
   P_n(a x)= \sum_{m=0}^n{n!\over m!}a^m\beta_{n-m}P_m(x),\quad a\neq 0,
\end{equation}
where ${A(t)\over A(at)}=\sum_{k=0}^\infty\beta_kt^k .$
\\
In fact, the PS $Q_n(x)=P_n(a x)$ is generated by
\begin{equation*}
  A(t)B(axt) =  \sum_{n=0}^{\infty}{Q_n(x)\over n!}t^n.
\end{equation*}
Thus, by using (\ref{def teta}) and (\ref{theta power}), we have  
$\theta(f)(x)=f(ax)$, where $f$ is any formal power series. 
\\
Now, from (\ref{generatrice cc}), with $A_1(t)=A_2(t)=A(t)$, it follows immediately that 
\begin{equation*}
 (at)^m{A(t)\over A(at)}=
     \sum_{n=m}^{\infty}\frac{m!}{n!}C_m(n)t^n.
\end{equation*}
\subsection{Linearization Problems}
In the following result, we provide a generating function for the LC involving three Brenke polynomials.
\begin{thm}\label{theo linearization}
  Let $\{P_n\}_{n\geq0}$, $\{R_n\}_{n\geq0}$ and $\{S_n\}_{n\geq0}$ be three Brenke PS with exponential generating functions:
\begin{equation}\label{3 brenke}
   A_1(t)B_1(xt),\ A_2(t)B_2(xt)\ \textrm{and}\ A_3(t)B_3(xt),
\end{equation}
where 
$A_i(t)=\sum_{k=0}^\infty a_k^{(i)}t^k,
\,
B_i(t)=\sum_{k=0}^\infty b_k^{(i)}t^k,\ $  $a_0^{(i)} b_k^{(i)} \neq 0, \,
\forall k\in \mathbb{N} ,\ 
 i=1,2,3.$ 
\\
 Then the LC, 
 $\{L_{ij}(k)\}_{i,j\geq0}, \, k \in \mathbb{N}$, 
 defined in 
 (\ref{linearization coeff}) are generated by:
 \begin{equation}\label{generalized linear}
{A_2(s)A_3(t)\over k!}\theta_s^{(2)}\theta_t^{(3)}(\theta_{s+t}^{(1)})^{-1}\left({(s+t)^k\over A_1(s+t)}\right)
= \sum_{ i,j \geq 0}{L_{ij}(k)\over i!j!}s^it^j
 \end{equation}
 where
 $\theta^{(i)}(t^n)=n!b_n^{(i)}t^n,\ \ i=1,2,3.$
\end{thm}
We note that 
$\theta^{(i)},\, i=1,2,3,$ 
are the transfer operators from  
$\{P_n\}_{n\geq0}$, $\{R_n\}_{n\geq0}$ and $\{S_n\}_{n\geq0}$, 
to the monomials, respectively.
\begin{proof}
On one hand, according to (\ref{linearization coeff}) and with sum manipulation, we obtain:
  \begin{align}\label{lin1}
    \sum_{ i,j\geq 0} R_i(x)S_j(x){s^i\over i!}{t^j\over j!} 
    &=& 
    \sum_{ i,j \geq 0}\left(
    \sum_{k=0}^{i+j}L_{ij}(k)P_k(x)\right)
    {s^i\over i!}{t^j\over j!}
    \nonumber 
    \\
     &=& \sum_{k=0}^\infty\left(k!\sum_{   i,j\geq 0}{L_{ij}(k)\over i!j!}s^it^j\right)
     {P_k(x)\over k!}.
  \end{align}
On the other hand, 
by using~(\ref{def teta}), we can easily verify that
$$ 
\theta_s^{(2)}\theta_t^{(3)}(\theta_{s+t}^{(1)})^{-1}B_1((s+t)x)=
 \sum_{k=0}^\infty\left(\sum_{l=0}^kb_l^{(2)}b_{k-l}^{(3)}s^lt^{k-l}\right)x^k, 
 $$
then
$$ 
 B_2(xs)B_3(xt) =
 \theta_s^{(2)}\theta_t^{(3)}(\theta_{s+t}^{(1)})^{-1}B_1((s+t)x).
 $$
Using the generating function of 
$\{P_n\}_{n\geq0}$, 
we obtain  
$$ 
B_2(xs)B_3(xt) = 
\sum_{k=0}^\infty\left(
   \theta_s^{(2)}\theta_t^{(3)}
   (\theta_{s+t}^{(1)})^{-1}
   {(s+t)^k\over A_1(s+t)}
   \right)
   {P_k(x)\over k!}.
   $$
Thus 
$$
 \sum_{ i,j\geq 0} R_i(x)S_j(x){s^i\over i!}{t^j\over j!}=
\sum_{k=0}^\infty\left(A_2(s)A_3(t)\theta_s^{(2)}\theta_t^{(3)}(\theta_{s+t}^{(1)})^{-1}{(s+t)^k\over A_1(s+t)}\right){P_k(x)\over k!}.
$$
Equating the coefficients of $P_k(x)$ in the above equation and (\ref{lin1}), we obtain (\ref{generalized linear}) which finishes the proof.
\end{proof}
Next, as applications, we recover the generating function for the LC of three Appell polynomials and the explicit expression of the LC associated to three Brenke PS.
\subsubsection{Appell Polynomials}
Let $\{P_n\}_{n\geq0}$, 
$\{R_n\}_{n\geq0}$, and 
$\{S_n\}_{n\geq0}$,
 be three Appell-PS. 
 Then we have
$B_1=B_2=B_3= \exp,$ 
and by applying 
Theorem~\ref{theo linearization}, 
we obtain that the LC in (\ref{linearization coeff}) are generated by
 \begin{equation}\label{linear appel poly}
  {A_2(s)A_3(t)\over A_1(s+t)}{(s+t)^k\over k!}= \sum_{i,j=0}^\infty{L_{ij}(k)\over i!j!}s^it^j,
 \end{equation}
which agrees with Carlitz Formula~\cite[Eq.(1.9)]{carlitz63}.
\\
Moreover, for $P_n=R_n=S_n=H_n$, 
where $H_n$ are Hermite polynomials generated by 
\begin{equation}\label{defn}
e^{-t^2}e^{2xt}=\sum_{n=0}^\infty
{H}_n(x)\frac{t^n}{n!},
\end{equation} 
we have 
$A_1(t)=A_2(t)=A_3(t)=A(t)=e^{-t^2}$,  and then 
$$
{A(s)A(t)\over A(s+t)}{(s+t)^k\over k!}=\frac{1}{k!}e^{2st}(s+t)^k.
$$
Thus, using (\ref{linear appel poly})
we deduce the standard linearization formula for Hermite PSs
\begin{equation}\label{feldheim}
H_i(x)H_j(x)=\sum_{k=0}^{\min(i,j)}{{i\choose
k}{j\choose k}2^k k!} H_{i+j-2k}(x).
\end{equation}
This formula is known as Feldheim formula~\cite{Askey}.
\subsubsection{Explicit Expression of the LC}
For three Brenke PS satisfying the hypothesises of Theorem~\ref{theo linearization}, the LC in (\ref{linearization coeff}) are given by:
\begin{equation}\label{explicit linear}
L_{ij}(k)=\frac{i!j!}{k!}
\sum_{n=0}^i\sum_{m=0}^j{b_n^{(2)}b_m^{(3)}\over b_{n+m}^{(1)}}
a_{i-n}^{(2)}a_{j-m}^{(3)}
\widehat{a}_{n+m-k}^{(1)},\quad 
k=0,1,\ldots,i+j,
\end{equation}
where  
$ {1/ A_1(t)}=\sum_{n=0}^\infty\widehat{a}_n^{(1)}t^n$, 
and 
\, $\widehat{a}_{-n}^{(1)}=0,\,
 n=1,2,\ldots.$
\\
Indeed, we have 
${(s+t)^k\over A_1(s+t)}=
\sum_{n=k}^\infty\widehat{a}_{n-k}^{(1)}(s+t)^n, 
$
then by using (\ref{def teta}), we get
$$
\theta_s^{(2)}\theta_t^{(3)} (\theta_{s+t}^{(1)})^{-1}
\left( {(s+t)^k\over A_1(s+t)} \right)= 
\sum_{n=k}^\infty\widehat{a}_{n-k}^{(1)}\sum_{m=0}^n{b_{n-m}^{(2)}b_{m}^{(3)}\over b_n^{(1)}}t^ms^{n-m}.
$$
Thus, with sum manipulations and (\ref{generalized linear}), one can easily verify that  
\begin{align*}
\sum_{i,j\geq0}{L_{ij}(k)\over i!j!}s^it^j
   &= {1\over k!}\sum_{n,m=0}^\infty\left(\sum_{i=n}^\infty a_{i-n}^{(2)}s^i\right)\left(\sum_{j=m}^\infty a_{j-m}^{(3)}t^j\right)
{b_n^{(2)}b_m^{(3)}\over b_{n+m}^{(1)}}\widehat{a}_{n+m-k}^{(1)}
 \\
   &= {1\over k!}\sum_{i,j\geq0}\left(\sum_{n=0}^i\sum_{m=0}^j
   {b_n^{(2)}b_m^{(3)}\over b_{n+m}^{(1)}}a_{i-n}^{(2)}a_{j-m}^{(3)}\widehat{a}_{n+m-k}^{(1)}\right)s^it^j,
\end{align*}
which leads to (\ref{explicit linear}).

We note that this result was first obtained 
in~\cite[Corollary 3.3]{chaggara2007c} 
by using a method based on the inversion formula.
\section{Application to Generalized Gould-Hopper Polynomial Set}
The $(d+1)$-fold symmetric generalized Gould-Hopper polynomials, \\
$\{Q_n^{(d+1)}(\cdot,a,\mu)\}_{n\geq0}$, are generated by~\cite{gaibch}:
\begin{equation}\label{Gouldhopper}
e^{at^{d+1}}\exp_{\mu}(xt)=
\sum_{n=0}^{\infty}
\frac{Q_n^{(d+1)}(x,a,\mu)}{n!}t^n,\ 
a\in\mathbb{C},\ \mu\neq -\frac{1}{2},-\frac{3}{2},-\frac{5}{2},\ldots,
\end{equation}
where a PS 
$\{P_n\}_{n\geq 0}$ 
is said to be $(d+1)$-fold symmetric, 
$d=1,2, \ldots,$
 if
$$
P_n\Bigl(e^{\frac{2i\pi}{d+1}}x\Bigr)=e^{\frac{2i n\pi}{d+1}}P_n(x).
$$
These polynomials constitute a unification of many known families such as:
\begin{itemize} 
\item 
Classical Hermite PS, $H_n(x)=Q_n^{(2)}(2x,-1,0)$.
\item 
Gould-Hopper PS,  
$g_n^{m}(x,h)=Q_n^{(m)}(x,h,0),$
~(same notations as in~\cite{gou}).
 \item 
 Generalized Hermite polynomials~\cite{szego75}:
\begin{equation}\label{rr}
 H_n^\mu(x)=Q_n^{(2)}(2x,-1,\mu).
 \end{equation}    
\end{itemize}
The GGHPS 
are of Brenke type with transfer power series
 $A(t)=\exp(a t^{d+1})$. They
are the only $(d+1)$-fold symmetric Dunkl-Appell $d$-orthogonal PS~\cite{gaibch}. 

Next, we solve the connection and linearization problems associated to
GGHPS and we treat the particular case of generalized Hermite polynomials.
\subsection{Connection Problem}
Here, we state the connection formulas for two GGHPS when one or two of the parameters are different and 
we give an integral representation of these coefficients. Moreover, the inversion formula, addition and convolution relations, and duplication formula are given.
\begin{thm}
The connection coefficients,
 $C_{n-i(d+1)}(n),\, 
 0\leq i\leq [\frac{n}{d+1}]$, between 
two GGHPS,
$\{Q_n^{(d+1)}(\cdot,a,\mu_1)\}_{n\geq0}$ and $\{Q_n^{(d+1)}(\cdot,b,\mu_2)\}_{n\geq0}$ are given by
  \begin{equation}\label{connection coef gld hpr}
    C_{n-i(d+1)}(n)=
    \frac{n!}{(n-i(d+1))!}
    \sum_{k=0}^i
    {\gamma_{\mu_1}(n-k(d+1))\over\gamma_{\mu_2}(n-k(d+1))}
    \frac{(-a)^{i-k}}{(i-k)!}\frac{b^k}{k!}. 
    \end{equation}
\end{thm}
\begin{proof}
By means of (\ref{def teta}), we have
$$
 \theta(t^m e^{-at^{d+1}} )
   =
     \sum_{n=0}^\infty
     {(-a)^n\over n!}{\gamma_{\mu_1}(n(d+1)+m)\over\gamma_{\mu_2}(n(d+1)+m)}t^{n(d+1)+m}.
$$
Thus, by using (\ref{generatrice cc}),  (\ref{Gouldhopper}) and sum manipulation, we obtain
  \begin{align*}
   \sum_{n=m}^\infty{m!\over n!}C_m(n)t^n &= e^{bt^{d+1}}
   \theta(t^m e^{-at^{d+1}} )
   \\
     &= \sum_{i=0}^\infty\frac{1}{i!}\sum_{k=0}^i {i\choose k}{\gamma_{\mu_1}(k(d+1)+m)\over\gamma_{\mu_2}(k(d+1)+m)}b^{i-k}(-a)^k\,t^{i(d+1)+m}.
  \end{align*}
Therefore, for
 $n=i(d+1)+m$, the desired result holds.
\end{proof} 
We note that for the particular case
$\mu_1=\mu_2$, 
(\ref{connection coef gld hpr}) 
is reduced to 
$$ 
C_{n-i(d+1)}(n)=\frac{n!(b-a)^i}{i!(n-i(d+1))!},\quad  0\leq i\leq \Bigl[\frac{n}{d+1}\Bigr].
$$
For the connection coefficients obtained in Theorem~\ref{connection coef gld hpr}, we have the following result. 
\begin{prop}
For $\mu_2>\mu_1>-\frac{1}{2}$, the connection coefficient given by (\ref{connection coef gld hpr}) has the following integral representation, 
\begin{align*}\label{rep intg cc GGHP}
   C_{n-i(d+1)}(n)&=
   \frac{n!\beta^{-1}(\mu_1+\frac{1}{2},\mu_2-\mu_1)}{i!(n-i(d+1))!
   }\times\\
   &
   \int_{-1}^{1}
   t^{n-i(d+1)}|t|^{2\mu_1}(b-at^{d+1})^i
   \frac{(1-t^2)^{\mu_2-\mu_1}}{{1-t}}\,dt.
  \end{align*}
\end{prop}
\begin{proof}
Using Proposition~\ref{rrr1} with 
$f(x)= x^{n-k(d+1)}$ and $x=1$, we obtain 
$$
{\gamma_{\mu_1}(n-k(d+1))\over\gamma_{\mu_2}(n-k(d+1))}=
\frac{1}{
\beta(\mu_1+\frac{1}{2},\mu_2-\mu_1)}
    \int_{-1}^{1}t^{n-k(d+1)}|t|^{2\mu_1}\frac{(1-t^2)^{\mu_2-\mu_1}}{1-t}dt.
    $$  
Substituting the above equation in  (\ref{connection coef gld hpr}), we get:
  \begin{align*}
     C_{n-i(d+1)}(n)
    & = 
     \frac{n!}{i!(n-i(d+1))!}\frac{1}{\beta(\mu_1+\frac{1}{2},\mu_2-\mu_1)}\times
     \\
     &\int_{-1}^{1}t^n|t|^{2\mu_1}\frac{(1-t^2)^{\mu_2-\mu_1}}{1-t} \left(\sum_{k=0}^{i}{i\choose k}(-a)^{i-k}(\frac{b}{t^{d+1}})^k\right)dt,
  \end{align*}
from which the desired result follows.    
\end{proof}
Next, we give some specific expansion relations associated to GGHPS.
\begin{enumerate}
\item[$\bullet$] 
\textit{Explicit and inversion formulas:}
The following explicit expression and inversion formula of 
$\{Q_n^{(d+1)}(\cdot,a,\mu\}_{n\geq0}$ can be easily derived from (\ref{explicit-inversion Brenke}):
\begin{equation}\label{explt exp gld hpr}
Q_n^{(d+1)}(x,a,\mu)=n!\sum_{k=0}^{[\frac{n}{d+1}]}{a^k\over k!\gamma_\mu(n-(d+1)k)}\,x^{n-(d+1)k},
\end{equation}
and
\begin{equation}\label{inv sqs gld hpr}
\frac{x^n}{\gamma_\mu(n)}=\sum_{k=0}^{[\frac{n}{d+1}]}{(-a)^k\over k!(n-(d+1)k)!}Q_{n-(d+1)k}^{(d+1)}(x,a,\mu).
\end{equation}
\item[$\bullet$] 
\textit{Addition and convolution relations:}
\begin{equation}\label{trans gld hpr}
T_y^\mu Q_n^{(d+1)}(x,a,\mu)=\sum_{k=0}^n{n!y^{n-k}\over k!\gamma_\mu(n-k)} Q_k^{(d+1)}(x,a,\mu),
\end{equation}
\begin{equation}\label{conv GHPS 2}
  2^{\frac{n}{d+1}}T_y^\mu Q_n^{(d+1)}\Bigl(
  2^{\frac{-1}{d+1}}x,a,\mu\Bigr)
  =
  \sum_{k=0}^n{n\choose k }Q_k^{(d+1)}(y,a,\mu)\,Q_{n-k}^{(d+1)}(x,a,\mu),
\end{equation}
where $T_y^\mu=\exp_\mu(yD_\mu).$

For $\mu=0$, 
this equation is reduced to the well-known Gould-Hopper convolution type relation~\cite{gou} 
and for $m=2,\, h=-1$, we recover the Runge formula for Hermite polynomials~\cite{rung1914}
\item[$\bullet$] 
\textit{Duplication formula:}
\begin{equation*}\label{dupl gld hpr}
Q_n^{(d+1)}(\alpha x,a,\mu)=n!\sum_{k=0}^{[\frac{n}{d+1}]}{\alpha^{n-k(d+1)}(1-\alpha^{d+1})^ka^k\over(n-k(d+1))!k!}Q_{n-k(d+1)}^{(d+1)}(x,a,\mu),\ \alpha\neq 0.
\end{equation*}
\end{enumerate}
\subsection{Linearization Formula}
Taking into account the $(d+1)$-fold symmetry property of the GGHPS, any LC $ L_{ij}(k)$,  in  (\ref{linearization coeff}) vanishes when $k\neq i+j-r(d+1)$.
Thus, according to 
(\ref{explicit linear}), the corresponding LC is given by:
\begin{align*}\label{linearization gld hpr}
L_{ij}(i+j-r(d+1))&=\frac{i!j!}{(i+j-r(d+1))!}\sum_{n=0}^{[\frac{i}{d+1}]}
\sum_{m=0}^{[\frac{j}{d+1}]}\frac{a_1^na_2^m(-a_3)^{r-m-n}}
{n!m!(r-m-n)!}\times\\
&\frac{\gamma_{
\mu_3}(i+j-(m+n)(d+1))}{\gamma_{\mu_1}(i-n(d+1))\gamma_{\mu_2}(j-r(d+1))},\ 0\leq r\leq\Bigl [\frac{i+j}{d+1}\Bigr].
\end{align*}
We remark that there is no difficulty in proving the corresponding formula for the linearization of any arbitrary number of GGHPSs. We have:
\begin{align*}
\prod_{s=1}^{N}Q_{i_s}^{(d+1)}(x,a_s,\mu_s)&=
\sum_{r=0}^{[\frac{ i_1+\cdots+i_N}
                 {d+1}]}
{i_1!\cdots i_N!
\over(i_1+\cdots+i_N-r(d+1))!}\times
\\
&\sum_{s_1=0}^{[\frac{i_1}{d+1}]}\cdots
\sum_{s_N=0}^{[\frac{i_N}{d+1}]}{a_1^{s_1}\cdots a_N^{s_N}
(-a_{N+1})^{r-s_1-\cdots-s_N}
\over s_1!\cdots s_N!(r-s_1-\cdots-s_N)!}\times
\\
&
{\gamma_{\mu_{N+1}}(i_1+\cdots +i_N-(d+1)(s_1+\cdots +s_N))\over\gamma_{\mu_1}(i_1-(d+1)s_1)\cdots\gamma_{\mu_N}(i_N-(d+1)s_N)}\times
\\
&
Q_{i_1+\cdots +i_N-r(d+1)}^{(d+1)}(x,a_{N+1},\mu_{N+1}).
\end{align*}
\subsection{Generalized Hermite Polynomials}
The generalized Hermite polynomials, 
$\{H_n^{\mu}\}_{n\geq0}$, are introduced by Szeg\"{o}~\cite{szego75}, then investigated by Chihara in his PhD Thesis \cite{chiharathesis} and further studied by many other authors          \cite{chaggara2007c,rosen}. 
They are generated by:
\begin{equation}
e^{-t^{d+1}}\exp_{\mu}(2xt)=
\sum_{n=0}^{\infty}
\frac{H_n^{\mu}(x)}{n!}t^n,\ 
\mu\neq -\frac{1}{2},-\frac{3}{2},-\frac{5}{2},\ldots.
\end{equation} 
\begin{prop} 
The following connection relation holds:
\begin{equation}\label{conction gener hermite}
   \widehat{H}_n^{\mu_2}(x)=\sum_{k=0}^{[n/2]}{(-1)^k\,4^k\over k!} (\mu_2-\mu_1)_k\widehat{H}_{n-2k}^{\mu_1}(x),\ \mu_2>\mu_1>-\frac{1}{2},
\end{equation}
where $\{\widehat{H}_n^{\mu_i}\}_n,\ i=1,2$ are the normalized generalized Hermite PS given by
 $$
 \widehat{H}_n^{\mu_i}(x)={\gamma_{\mu_i}(n)\over n![\frac{n}{2}]!}{H_n^{\mu_i}(x)}.
 $$
\end{prop}
\begin{proof}
From what has already been stated, the connection coefficients  from $\{H_n^{\mu_2}\}_n$ to  $\{H_n^{\mu_1}\}_n$ are generated by 
\begin{equation*}
   e^{-t^2}\theta(t^me^{t^2})=\sum_{n=m}^{\infty}\frac{m!}{n!}C_m(n)t^n,
\end{equation*}
  where $\theta$ is the operator defined in (\ref{thetadunkl}).
  \\
Making use of the $\theta$-integral representation (\ref{reps intg teta}), intercalate 0 in the interval of integration, we get:
\begin{align*}
   \sum_{n=m}^{\infty}\frac{m!}{n!}C_m(n)t^n& =
    {t^me^{-t^2}\over \beta(\mu_1+\frac{1}{2},\mu_2-\mu_1)}\times\\
  &  \int_{0}^{1}e^{t^2s^2}
    {s^{m+2\mu_1}\over(1-s^2)^{\mu_1-\mu_2}}
    \left({1\over1-s}+{(-1)^m\over1+s}\right)\,ds.
\end{align*}
It follows, for $m$ even and after substituting $u=s^2$, that
  \begin{align*}
     \sum_{n=m}^{\infty}\frac{m!}{n!}C_{m}(n)t^n   
     &={t^{m}e^{-t^2}\over \beta(\mu_1+\frac{1}{2},\mu_2-\mu_1)}\int_{0}^{1}e^{ut^2}u^{\frac{m-1}{2}+\mu_1}(1-u)^{\mu_2-\mu_1-1}du \\
     &= \sum_{n=0}^{\infty}{(-1)^n\over n!}{\beta(\mu_1+\frac{m+1}{2},\mu_2-\mu_1+n)\over\beta(\mu_1+\frac{1}{2},\mu_2-\mu_1)}t^{m+2n},
  \end{align*}
where the term by term integration is justified by the same argument as in the proof of Proposition \ref{prophyper}.
  \\
On the other hand, we have
   \begin{align*}
    {\beta(\mu_1+\frac{1}{2}+k,\mu_2-\mu_1+n)\over\beta(\mu_1+\frac{1}{2},\mu_2-\mu_1)} &={\Gamma(\mu_1+\frac{1}{2}+k)\Gamma(\mu_2-\mu_1+n)\Gamma(\mu_2+\frac{1}{2})\over\Gamma(\mu_2+n+k+\frac{1}{2})\Gamma(\mu_1+\frac{1}{2})\Gamma(\mu_2-\mu_1)}\\
     &={\gamma_{\mu_1}(2k)\over2^{2k}k!}{2^{2(k+n)}(k+n)!\over\gamma_{\mu_2}(2(k+n))}(\mu_2-\mu_1)_n\\
     &={\gamma_{\mu_1}(m)\over\gamma_{\mu_2}(m+2n)}{4^n([m/2]+n)!\over[m/2]!}(\mu_2-\mu_1)_n.
  \end{align*}
Thus, by virtue of (\ref{beta}) and (\ref{fracdunkl}), we obtain
 \begin{equation*}\label{connection rela even}
     \sum_{n=m}^{\infty}\frac{m!}{n!}C_m(n)t^n=\sum_{n=0}^{\infty}{(-1)^n\over n!}{\gamma_{\mu_1}(m)4^n([\frac{m}{2}]+n)!\over\gamma_{\mu_2}(m+2n)[\frac{m}{2}]!}(\mu_2-\mu_1)_nt^{m+2n}.
  \end{equation*}
 For $m$ odd, similar computations lead to
 \begin{equation*}
     \sum_{n=m}^{\infty}\frac{m!}{n!}C_{m}(n)t^n =\sum_{n=0}^{\infty}{(-1)^n\over n!}{\gamma_{\mu_1}(m)\over\gamma_{\mu_2}(m+2n)}{4^n([\frac{m}{2}]+n)!\over[\frac{m}{2}]!}(\mu_2-\mu_1)_nt^{m+2n}.
  \end{equation*}
Therefore, for $m=0,1,2,3,\ldots $, we have:
\begin{equation*}
   \sum_{n=m}^{\infty}\frac{m!}{n!}C_m(n)t^n =\sum_{n=0}^{\infty}{(-1)^n\over n!}{\gamma_{\mu_1}(m)\over\gamma_{\mu_2}(m+2n)}{4^n([\frac{m}{2}]+n)!\over[\frac{m}{2}]!}(\mu_2-\mu_1)_nt^{m+2n},
\end{equation*}
Thus, for $\ k=0,1,2,\ldots,[\frac{n}{2}],$ we get 
\begin{equation*}
  C_{n-2k}(n)={(-1)^k\over k!}{n!\over (n-2k)!}{4^k[\frac{n}{2}]!\over
  [\frac{n}{2}-k]!}{\gamma_{\mu_1}(n-2k)\over\gamma_{\mu_2}(n)}(\mu_2-\mu_1)_k.
\end{equation*}
\end{proof}
We note that the connection coefficients in (\ref{conction gener hermite}) alternate in sign and that this relation was already derived in \cite{chaggara2011}, where the authors used a linear computer algebra approach based on the Zeilberger's algorithm.

\end{document}